\documentclass[11pt, a4paper]{article}

\usepackage{upref,amsmath,amssymb,amsthm}
\usepackage{amsmath}
\usepackage{graphicx}
\usepackage[swedish,english]{babel}

\usepackage{amsmath,amssymb,upref,amsthm,eucal}
\usepackage{graphics}
\usepackage[swedish,english]{babel}
\newtheorem{theorem}{Theorem}
\newtheorem{lemma}[theorem]{Lemma}

\newtheorem{proposition}[theorem]{Proposition}
\theoremstyle{remark}

\newcommand{\XB}{\mathbf{X}}
\newcommand{\YB}{\mathbf{Y}}
\newcommand{\UB}{\mathbf{U}}
\newcommand{\VB}{\mathbf{V}}

\begin{document}

\title{A consistent estimator of the smoothing operator in the functional Hodrick-Prescott filter}

\author{  Hiba Nassar\thanks{School of Computer Science, Physics and Mathematics, Linnaeus University, Vejdesplats 7, SE-351 95 V\"axj\"o, Sweden. e-mail: hiba.nassar@lnu.se}}

\date{\today}
 \maketitle
 
\begin{abstract}
In this paper we consider a version of the functional Hodrick-Prescott filter for functional time series.  We show that the associated optimal smoothing operator preserves the 'noise-to-signal' structure. Moreover, we propose a consistent estimator of this optimal smoothing operator.

\end{abstract}
\bigskip
\noindent
\small{ {\it JEL classifications}: C5, C22, C13, E32.}
\medskip

\noindent \small{ {\it AMS 2000 subject classifications}: 62G05, 62G20.}

\medskip
\noindent
\small{{\it Key words and phrases}: Inverse problems, adaptive estimation, Hodrick-Prescott filter, functional time series, smoothing, signal extraction, Hilbert space-valued Gaussian random variables.}

\section{Introduction}
The study of functional time series has attracted a large amount of research in the functional data analysis literature (see for example, H\"ormann and Kokoszka (2010), Horv\'ath et al. (2010), Horv\'ath and Kokoszka (2012)). Functional time series can be obtained by splitting an almost continuous time record of measurements (functional data) into natural consecutive intervals, and the measurements in each interval are treated as a whole observational unit.

\medskip\noindent For simplicity, a functional time series can be seen as a sample consists of $n$  curves $X_1, X_2, ... , X_n$ belonging to the space $L^2([0,1])$ of square integrable functions on $[0,1]$, where $X_i  = \{X_i (t );  0 \le  t  \le 1 \}$, $i  = 1,2,..,n$; $i $ refers to a day, a month or a year, and $t$  is the time within that unit interval representing a sufficiently dense grid of times (tics, seconds, etc.) at which an observation $X_n(t)$ is made.

\medskip\noindent 
The classical Hodrick-Prescott filter (called henceforth the HP filter) was introduced by Hodrick and Prescott (1997) and since then it has been used widely in economics and actuarial science.  The classical Hodrick-Prescott filter was proposed  as a procedure of extracting a 'signal' (also called trend in the economic literature) $y(\alpha,x)=(y_1(\alpha,x),\ldots, y_T(\alpha,x))$ from a  real-valued, noisy time series $x=(x_1,\ldots,x_T)$ and an appropriately chosen positive parameter $\alpha$, called the smoothing parameter. They suggested that the real-valued time series $(x,y)$  satisfies the following linear mixed model:
\begin{equation}\label{M-HP}
\left\{\begin{array}{lll}
x=y+u,\\
Py=v,
\end{array}\right.
\end{equation}
where $u\sim N(0,\sigma_u^2I_T)$ and $v\sim N(0,\sigma_v^2I_{T-2})$ ($I_T$ and $I_{T-2}$ denote the $ T \times T$ and $(T-2) \times (T-2)$ identity matrices, respectively) and $P$ is the discrete second order differencing operator $(Py)(t):=y_{t +2}-2 y_{t +1}+ y_{t}, \, t=1,\ldots,T -2$. For an appropriate smoothing parameter, the 'optimal smooth' signal associated with $x$ is the minimizer of the following functional
\begin{equation}\label{HP}
\sum_{t=1}^T(x_t-y_t)^2+\alpha\sum_{t=1}^{T-2}( y_{t +2 }-2 y_{t+1}+ y_{t})^2,
\end{equation}
with respect to $y=(y_1,\ldots,y_T)$. Using the model (\ref{M-HP}) above, the optimal smoothing parameter turns out to be the so-called 'noise-to-signal ratio', i.e. $\alpha^*=\sigma_u^2/\sigma_v^2$.  Schlicht (2005) proved that the noise-to-signal ratio satisfies
\begin{equation}\label{opt-alpha-bis}
E[\, y|\,x]= y(\frac{\sigma_u^2}{\sigma_v^2},x)
\end{equation}
and is optimal in the sense that (see Dermoune {\it et al.} (2009)) the optimal smoothing parameter  minimizes the mean square difference between the 'optimal signal' $y(\alpha,x)$ and the conditional expectation $E[\, y|\,x]$ which is the best predictor of any signal $y$ given the time series $x$, namely
 \begin{equation}\label{opt-alpha}
\sigma_u^2/\sigma_v^2=\arg\min_{\alpha}\left\{\|E[\, y|\,x]-y(\alpha,x)\|^2\right\}.
\end{equation}
Moreover, Dermoune {\it et al.} (2009) derived an explicit unbiased consistent estimator of the smoothing parameter. They dealt with the observations from the Gaussian time series $Px$: 
$$
Px = v + Pu \sim N(0,\sigma_v^2I_{T-2} + \sigma _u ^2 P P^{'}),
$$
to suggest the following consistent estimators of the variances $\sigma_u^2$ and $\sigma_v^2$: 

\begin{equation}
\label{sigmauest}
\hat{\sigma}_u^2= \frac{1}{4(T-3)} \sum_{j=1}^{T-3} { Px(j) Px(j+1)}, 
\end{equation}

and
\begin{equation}
\label{sigmavest}
\hat{\sigma}_v^2= \frac{1}{(T-2)} \sum_{i=1}^{T-2} { Px(j) ^2} +\frac{3}{2(T-3)} \sum_{j=1}^{T-3} { Px(j) Px(j+1)}.
\end{equation}

\medskip\noindent Due to the noise-to-signal ratio, the optimal smoothing parameter $\alpha^*$ admits the following consistent estimator 
\begin{equation}
\label{alphaest}
\hat{\alpha}= - \frac{1}{4} \left( \frac{3}{2} + \frac{(T-3) \sum_{j=1} ^{T-2} { Px(j) ^2}}{ (T-2) {\sum_{j=1}^{T-3} { Px(j) Px(j+1)}}} \right)^{-1}.
\end{equation}

\medskip\noindent  Djehiche and Nassar (2013)  suggested a  functional  version of the Hodrick-Prescott filter for which the data $x$ take values in a possibly infinite dimensional Hilbert space. The functional HP filter is described as a mixed model of the same form as  (\ref{M-HP}), where the second backward difference operator $P$ is replaced by a compact operator $A$.  They characterized the optimal smoothing parameter determined by a criterion similar to (\ref{opt-alpha})  where the noise $u$ and the  signal $v$ are independent Hilbert space-valued Gaussian random variables with zero means and covariance operators $\Sigma_u$ and $\Sigma_v$.

\medskip\noindent Furthermore, Djehiche {\it et al.} (2013) extended the functional Hodrick-Prescott filter  to the case where the operator $A$ is closed and densely defined with closed range.

\medskip\noindent  In this paper, we extend the functional Hodrick-Prescott filter to the case of functional time series, using the second order differencing operator. We also characterize the optimal smoothing operator, based on the optimality criterion suggested in Djehiche and Nassar (2013). Moreover, we suggest  an explicit and consistent estimator of the optimal smoothing operator.
 
\medskip\noindent
The paper is organized as follows. In Section 2, we introduce the Hodrick-Prescott filter for functional time series. In Section 3, we prove that the optimal smoothing operator preserves the noise-to-signal ratio structure. In Section 4, we propose a consistent estimator of the optimal smoothing operator.
 
\medskip\noindent  Functional time series such as: intradaily financial transactions, Geophysical data, magnetometer data (for further details, see \cite{Kokoszka1})  are good examples to apply this study.

\section{The functional Hodrick-Prescott filter for functional time series } 
In this section, we propose the functional Hodrick-Prescott filter for functional time series, where the observations $\XB$ can be seen as a vector of $n$ entries of functional data.

\medskip\noindent Let $\XB = \left( \XB_1, \XB_2, ..., \XB_n\right)$ be a functional time series, where $ \XB_i \in L^2([0,1]) $ for $i= 1,2,...,n$. The functional Hodrick-Prescott filter for functional time series is a procedure to reconstruct an 'optimal smooth signal' $\YB = \left( \YB_1, \YB_2, ..., \YB_n\right)$ that solves an equation  
\begin{equation}
P \YB _m=  \YB_ {m+2} -2 \YB _{m+1} + \YB_m = \VB _m, \qquad m= 1,2,..., n-2,
\end{equation}
 corrupted by a noise $\VB = \left( \VB_1, \VB_2, ..., \VB_{n-2}\right)$  which is apriori unobservable, from observations $\XB$ corrupted by a noise $\UB = \left( \UB_1, \UB_2, ..., \UB_n\right)$ which is also apriori unobservable:
\begin{equation}\label{F-HP}
\left\{\begin{array}{lll}
\XB=\YB+\UB,\\
 P \YB=\VB.
\end{array}\right.
\end{equation}

\medskip\noindent  The second order backward shift operator $P$ can be written in vector form as the following $ (n-2) \times  n$-matrix 

\[ \left ( \begin{matrix}
  1 & -2 & 1 & 0 &...& ...& 0 \\
  0 & 1 & -2 & 1 &... & ...& 0 \\
	0 & 0 & 1 & -2 & 1  & ...& 0 \\
  	&\ldots  &  \ldots &  \ldots &  \ldots & & \\
  0 & 0 & 0 & 0 & 1 & -2 & 1
 \end{matrix}
\right). \]

\medskip\noindent Following \cite{djehiche3}, let $B: L^2([0,1]) \rightarrow L^2([0,1])$ be a smoothing operator, which is  linear, bounded and positive. The optimal smooth signal, associated with $\XB$, is obtained by regularizing the system (\ref{F-HP}): 
\begin{equation}\label{HP-H-trend}
\YB(B,\XB) := \arg \min_\YB \left\{ \sum_{i=1}^n {\left\| \XB_i -\YB_i \right\|^2_{ L^2([0,1])}} + \sum_{i=1}^{n-2} { \langle (P \YB)_i, B (P \YB)_i \rangle_{ L^2([0,1])}} \right\},
\end{equation} 
provided that 
$$
\langle h , B h  \rangle_{L^2([0,1])} \ge 0, \qquad h\in L^2([0,1]).
$$

\medskip\noindent To find the optimal smoothing operator, we will use the selection criterion in \cite{djehiche3}, namely: 
\begin{equation}\label{B-best}
\hat B = \arg \min_B \left\| E[\YB|\XB] - \YB(B ,\XB) \right\|^2.
\end{equation}

\medskip\noindent This selection criterion minimizes the difference between the optimal solution $\YB(B , \XB)$, and the conditional expectation $E[\YB|\XB]$.

\medskip\noindent The main aim of this work is to derive a consistent estimator of the optimal smoothing operator $B$, extending the results of \cite{djehiche1} to infinite dimensional case.

\medskip\noindent Let $\{e_1, e_2, ...,e_j, ...\}$ be an orthogonal basis in $L^2([0,1])$ (a well known example of a basis in $L^2([0,1])$ is $ \{ e_i(t) = \sqrt {2}\sin (i \pi t )\},  i= 1,2,...$). $\XB_i, \YB_i, \UB_i$ admit the following representation for $i= 1,2,...,n$, respectively:
$$
\XB_i = \sum _{j=1} ^\infty  {x_{i}^j e_j }, \qquad \YB_i = \sum _{j=1} ^\infty  {y_{i} ^j e_j }, \qquad \UB_i = \sum _{j=1} ^\infty  {u_{i} ^j e_j },
$$
and $\VB_i$ admits similar representation for $i= 1,2,...,n-2:$
$$
\VB_i = \sum _{j=1} ^\infty  {v_{i}^j e_j }.
$$
\medskip\noindent For arbitrary $j$  the projectors of  $\XB, \YB, \UB$ and $\VB$ onto the eigenspace $\overline {\mbox {Span} \{e_j\} }$, can be seen as the following vectors in $\mathbb{R}^n$ and $\mathbb{R}^{n-2}$, respectively:
$$ 
\bar{X}^j = (x_{1}^j, x_{2}^j, ..., x_{n}^j ), \qquad \bar{Y}^j = (y_{1}^j, y_{2}^j, ..., y_{n}^j ), \qquad \bar{U}^j = (u_{1}^j, u_{2}^j, ..., u_{n}^j ), 
$$
and 
$$
\bar{V}^j = (v_{1}^j, v_{2}^j, ..., v_{n-2}^j ).
$$
Hence 
$$ 
\XB = \sum_{j=1} ^\infty {\bar{X}^j \vec{e}^{(n)}_j}, \qquad \YB = \sum_{j=1} ^\infty {\bar{Y}^j \vec{e}^{(n)}_j}, \qquad \UB = \sum_{j=1} ^\infty {\bar{U}^j \vec{e}^{(n)}_j}, 
$$
and 
$$
\VB = \sum_{j=1} ^\infty {\bar{V}^j \vec{e}^{(n-2)}_j},
$$
where, 
$$\vec{e}^{(n)}_j = \left( (e_j,0,...,0)_{1\times n}^{'}, (0,e_j,...,0)_{1\times n}^{'}, ..., (0,0,...,e_j)_{1\times n}^{'} \right)_{1 \times n},$$
 and 
 $$\vec{e}^{(n-2)}_j = \left( (e_j,0,...,0)_{1\times n-2}^{'}, (0,e_j,...,0)_{1\times n-2}^{'}, ..., (0,0,...,e_j)_{1 \times n-2}^{'} \right)_{1\times n-2}.$$
\medskip\noindent Moreover, the system (\ref{F-HP}) yields that for every $j$ 
\begin{equation}\label{P-HP}
\left\{\begin{array}{lll}
\bar{X}^j=\bar{Y}^j+ \bar{U}^j,\\
 P \bar{Y}^j =\bar{V}^j,
\end{array}\right.
\end{equation}
where $P$ is  the real second order differencing operator 
\begin{equation}
P \bar{Y}^j_m := y_{m+2}^j -2 y_{m+1}^j +y_m^j ; \qquad   m = 1, \ldots , n-2. 
\end{equation}

\medskip\noindent Since the smoothing operators $B$ is linear and  bounded, by Riesz' Representation Theorem, there exist uniquely determined $ \alpha_j>0,\, j=1,2,\ldots$, such that
\begin{equation}\label{smoothing}
 B h = \sum_{j=1} ^\infty {\alpha_j \langle h,e_j\rangle e_j},\qquad h\in L^2([0,1]),
 \end{equation} 
where the sum converges in the operator norm.

\medskip\noindent Following Dermoune {\it et al.} (2008), the optimal smoothing signal associated with $\bar{X}^j$ can be found as in (\ref{HP}) i.e.

\begin{equation}\label{HP-j}
\bar{Y}^j( \alpha_j, \bar{X}^j) =\arg \min_{y^j} \left\{ \sum_{i=1}^n(x_i^j-y_i^j)^2+\alpha_j\sum_{i=1}^{n-2}( y_{i+2}^{ j}-2 y_{i+1}^{j}+ y_i^{j})^2\right\}.
\end{equation}
The minimizer is given by the formula
\begin{equation}
\bar{Y}^j( \alpha_j, \bar{X}^j) = (y_1^j( \alpha_j, \bar{X}^j),y_2^j( \alpha_j, \bar{X}^j),...,y_n^j( \alpha_j, \bar{X}^j))= ( I_ {n} + \alpha_j P^{'} P )^{-1} \bar{X}^j.
\end{equation}
We have
\begin{equation}
\label{YBapprox}
\sum_{j=1}^\infty {\bar{Y}^j( \alpha_j, \bar{X}^j) \vec{e}^{(n)}_j} = ( I_ {H^{\otimes n}} + B P ^{'} P )^{-1} \XB.
\end{equation}
Set
\begin{equation}
\label{YBapproxfinal}
\YB(B, \XB)= ( I_ {H^{\otimes n}} + B P ^{'} P )^{-1} \XB.
\end{equation}
\begin {proposition}
Assume $B:L^2([0,1]) \rightarrow L^2([0,1])$ satisfies 
$$
\langle h , B h  \rangle_{L^2([0,1])} \ge 0, \qquad h\in L^2([0,1]).
$$ 
then (\ref{YBapproxfinal}) is the optimal smoothing signal which minimizes the functional 
\begin{equation}
\label{functionalmini}
J_B(\YB)=\sum_{i=1}^n {\left\| \XB_i -\YB_i \right\|^2_{ L^2([0,1])}} + \sum_{i=1}^{n-2} { \langle (P \YB)_i, B (P \YB)_i \rangle_{ L^2([0,1])}}.
\end{equation}
\end{proposition}
\begin{proof}
To prove $\YB ( B,\XB)$ is the minimizer of the functional $J_B(\YB)$ we will use the fact that $ \bar{Y}^j( \alpha_j, \bar{X}^j)$ is the minimizer of the functional 
$$
\sum_{i=1}^n(x_i^j-y_i^j)^2+\alpha_j\sum_{i=1}^{n-2}( y_{i+2}^{ j}-2 y_{i+1}^{j}+ y_i^{j})^2,
$$
for each $j$.
By a simple computation we have
\[
\begin{split}
J_B(\YB)& = \sum_{i=1}^n {\left\| \XB_i -\YB _i \right\|^2 _{ L^2([0,1])}} + \sum_{i=1}^{n-2} { \langle (P \YB)_i, B (P \YB)_i \rangle_{ L^2([0,1])}}\\
& = \sum_{i=1}^n { \sum_{j=1}^\infty {\left( x^j_i  -y_i^j \right)^2 }}  + \sum_{i=1}^{n-2} { \langle \sum_{j=1} ^\infty {\left(y_{i+2}^j -2 y_{i +1}^j +y_{i}^j\right)e_j}, \sum_{k=1} ^\infty { \alpha_k \left(y_{i+2}^k -2 y_{i +1}^k +y_{i}^k\right)e_k} \rangle}\\ 
& = \sum_{i=1}^n { \sum_{j=1}^\infty {\left( x_i^j  -y_i^j  \right)^2 }}  + \sum_{i=1}^{n-2} { \sum_{j=1} ^\infty { \alpha_j( y_{i+2}^{ j}-2 y_{i+1}^{j}+ y_i^{j})^2 }}\\
& \geq \sum_{j=1} ^\infty {\left( \sum_{i=1}^n(x_i^j-y_i^j( \alpha_j, X^j))^2+\alpha_j\sum_{i=1}^{n-2}( y_{i+2}^{ j}( \alpha_j, X^j)-2 y_{i+1}^{j}( \alpha_j, X^j)+ y_i^{j}( \alpha_j, X^j))^2 \right)}\\
& \geq J_B(\YB ( B,\XB)).
\end{split}
\]
\end{proof}

\section{Optimality of the noise-to-signal ratio}

We will now characterize the optimal smoothing operator, defined by (\ref{B-best}), associated with the Hodrick-Prescott filter (\ref{HP-H-trend}).

\medskip\noindent In (\ref{F-HP}), the equation $P \YB =\VB$ has a solution of the form 
\begin{equation}
\label{YB}
\YB =Y_0 + P ^{'} (P  P ^{'}) ^{-1} \VB ,
\end{equation}
where $Y_0 = Z \gamma$ such that the $n \times 2$- matrix $Z$ satisfies
$$
P Z = 0, \qquad Z^{'} Z= I_2
$$ 
with $\gamma \in L^2([0,1]) \times L^2([0,1])$ can be chosen arbitrarily, and 
\begin{equation}
\label{XB}
\XB  =Y_0 + P  ^{'} (P  P ^{'}) ^{-1} \VB + \UB.
\end{equation}
From (\ref{YB}) and (\ref{XB}), a stochastic model for $(\XB,\YB)$ is determined by models for $Y_0$ and $(\UB,\VB)$. 
 We assume
 under the following assumptions.


\medskip
{\bf Assumption (1)} $Y_0$ is deterministic.

\medskip
{\bf Assumption (2)} Let $\UB_1, \UB_2, ..., \UB_n$ be independent and identically distributed Hilbert space-valued Gaussian random variables with zero mean and covariance operator $\Sigma _u : L^2([0,1]) \rightarrow L^2([0,1])$, where the covariance operator $\Sigma_u$ is a linear, self-adjoint, positive-definite and trace class operator on $L^2([0,1])$.


\medskip
{\bf Assumption (3)} Let $\VB_1, \VB_2, ..., \VB_{n-2}$ be independent and identically distributed Hilbert space-valued Gaussian random variables with zero mean and covariance operator $\Sigma _v : L^2([0,1]) \rightarrow L^2([0,1])$, where the covariance operator $\Sigma_v$ is a linear, self-adjoint, positive-definite and trace class operator on $L^2([0,1])$.


\medskip\noindent Since the covariance operator $\Sigma_u$ is trace class and thus compact, by Riesz' Representation Theorem, there exist uniquely determined $ \mu_k>0,\, k=1,2,\ldots$, such that
\begin{equation}\label{sigma-u}
\Sigma _u h = \sum _{j=1} ^\infty  { \mu _j \langle h, e_j \rangle e_j, \qquad h \in L^2([0,1])},
 \end{equation}
where, the sum converges in the operator norm. 
Similarly for $\Sigma_v$, there exists uniquely determined $ \tau_k>0,\, k=1,2,\ldots$, such that
\begin{equation}\label{sigma-v}
\Sigma _v h = \sum _{j=1} ^\infty  {\tau_j \langle h, e_j \rangle e_j}, \qquad h \in L^2([0,1]),
 \end{equation}
where, the sum converges in the operator norm.

\medskip\noindent For each $j$, $\bar{U}_j $ is a real valued Gaussian random variable with zero mean and covariance operator $\Sigma _{\bar{U}_j} : \mathbb{R}^n \rightarrow \mathbb{R}^n$, where the covariance operator $\Sigma_{\bar{U}_j} = \mu_j I _n$. In terms of $\Sigma_{\bar{U}_j}$, we have $$\Sigma_\UB = \sum_{j=1} ^\infty {\Sigma_{\bar{U}_j} \vec{e}^{(n)}_j}.$$ 
Similarly, $\bar{V}_j $ is a real valued Gaussian random variable with zero mean and covariance operator $\Sigma _{\bar{V}_j} : \mathbb{R}^n \rightarrow \mathbb{R}^n$, where the covariance operator $\Sigma_{\bar{V}_j} = \tau_j I _{n-2}$. 


\medskip
{\bf Assumption (4)} Let $ \UB_r$ and $\VB_s$ be pairwise independent for all $r = 1,2,...,n$ and $s = 1,2,..., n-2$. 
From this assumption, it follows that $\bar{U}^j$ and $\bar{V}^j$ are independent real random variables for every $j$. 

\medskip\noindent Given Assumptions (1), (2), (3) and (4), in view of (\ref{YB}) and (\ref{XB}), it holds that $(\XB, \YB)$ is Gaussian with mean $(E(\XB),E( \YB)) = (Y_0,Y_0)$, and covariance operator 
\begin{equation}
\label{cov}
\Sigma =
 \begin{pmatrix}
  \Sigma _\XB & \Sigma _{\XB \YB} \\
  \Sigma _{\YB \XB} &\Sigma _\YB
 \end{pmatrix},
\end{equation}
where $ \Sigma _\XB, \Sigma _{\XB \YB}$ and $\Sigma _\YB$ are operators on $H^{\otimes n}$, with $\Sigma _{\XB \YB}  = \Sigma _{\YB \XB}. $

\medskip\noindent 
\begin{proposition}
Under Assumptions (1), (2), (3) and (4), we have
\[ \langle \langle \Sigma_ \XB \vec{e}^{(n)}_j ,\vec{e}^{(n)}_k \rangle \rangle = \Sigma_{\bar{X} ^j} \delta _{jk},  \quad \langle \langle \Sigma_ \YB \vec{e}^{(n)}_j ,\vec{e}^{(n)}_k \rangle \rangle = \Sigma_{\bar{Y} ^j} \delta _{jk}\]
 and 
\[ \langle \langle \Sigma_ {\XB \YB} \vec{e}^{(n)}_j ,\vec{e}^{(n)}_k \rangle \rangle = \Sigma_{\bar{X} ^j \bar{Y} ^j} \delta _{jk}.\]
with 
\begin{equation}
\label{sigmaXj}
\Sigma_{\XB} =  \Sigma _{\UB} + \Sigma_{\VB} P^{'}( P P^{'})^{-1} (P P^{'})^{-1} P 
\end{equation}
\begin{equation}
\label{sigmaYj}
\Sigma_{\YB} = \Sigma_{\VB} P^{'}( P P^{'})^{-1} (P P^{'})^{-1} P,
\end{equation}
and 
\begin{equation}
\label{sigmaXYj}
\Sigma_{\XB \YB} = \Sigma_{\VB} P^{'}( P P^{'})^{-1} (P P^{'})^{-1} P.
\end{equation}
are trace class operators.
\end{proposition}

\begin{proof}
For each $j$, equations (\ref{YB}) and (\ref{XB}) have the form

\begin{equation}
\label{barYj}
\bar{Y}^j =Y_0^j  +  P ^{'} (P  P ^{'}) ^{-1} \bar{V}^j ,
\end{equation}
and 
\begin{equation}
\label{barXj}
\bar{X}^j =Y_0^j  + P ^{'} (P  P ^{'}) ^{-1} \bar{V}^j + \bar{U}^j.
\end{equation}
Due to the independency assumptions, we have for every $j \neq k$:
\[
\begin{split}
\mbox{cov } (\bar{X}^j, \bar{X}^k) &= \mbox{cov } (Y_0^j  + P ^{'} (P  P ^{'}) ^{-1} \bar{V}^j + \bar{U}^j , Y_0^k  + P ^{'} (P  P ^{'}) ^{-1} \bar{V}^k + \bar{U}^k)=0,\\
\mbox{cov } (\bar{Y}^j, \bar{Y}^k) &= \mbox{cov } (Y_0^j  + P ^{'} (P  P ^{'}) ^{-1} \bar{V}^j , Y_0^k  + P ^{'} (P  P ^{'}) ^{-1} \bar{V}^k )=0,\\
\mbox{cov } (\bar{X}^j, \bar{Y}^k) &= \mbox{cov } (Y_0^j  + P ^{'} (P  P ^{'}) ^{-1} \bar{V}^j + \bar{U}^j , Y_0^k  + P ^{'} (P  P ^{'}) ^{-1} \bar{V}^k )=0.
\end{split}
\]
and
\begin{equation}
\begin{split}
\label{barxj}
 \Sigma_{\bar{X}^j}&= \mbox{cov } (Y_0^j  + P ^{'} (P  P ^{'}) ^{-1} \bar{V}^j + \bar{U}^j , Y_0^j  + P ^{'} (P  P ^{'}) ^{-1} \bar{V}^j + \bar{U}^j) \\
& = \Sigma_{\bar{V}^j} P ^{'} (P  P ^{'}) ^{-1} (P  P ^{'}) ^{-1}  P ^{'} +\Sigma_{\bar{U}^j},
\end{split}
\end{equation}
\begin{equation}
\label{baryj}
\begin{split}
\Sigma_{\bar{Y}^j} &= \mbox{cov } (Y_0^j  + P ^{'} (P  P ^{'}) ^{-1} \bar{V}^j , Y_0^k  + P ^{'} (P  P ^{'}) ^{-1} \bar{V}^k )\\
& = \Sigma_{\bar{V}^j} P ^{'} (P  P ^{'}) ^{-1} (P  P ^{'}) ^{-1}  P ^{'}, \\
\end{split}
\end{equation}
\begin{equation}
\label{barxyj}
\begin{split}
\Sigma_{\bar{X}^j \bar{Y}^j} &= \mbox{cov } (Y_0^j  + P ^{'} (P  P ^{'}) ^{-1} \bar{V}^j + \bar{U}^j , Y_0^k  + P ^{'} (P  P ^{'}) ^{-1} \bar{V}^k )\\
&= \Sigma_{\bar{V}^j} P ^{'} (P  P ^{'}) ^{-1} (P  P ^{'}) ^{-1}  P ^{'}.
\end{split}
\end{equation}
Then 
\begin{equation}
\begin{split}
\langle \langle \Sigma_ \XB \vec{e}^{(n)}_j ,\vec{e}^{(n)}_k \rangle \rangle &= E \left[\langle X -Y_0,\vec{e}^{(n)}_j \rangle \langle X-Y_0,\vec{e}^{(n)}_k \rangle \right]\\
&= E \left[(\bar{X}^j-Y_0^j) (\bar{X}^k - Y_0^k)\right] \\
&= \left\{\begin{array}{lll}
\mbox{cov } (\bar{X}^j, \bar{X}^k) =0, \quad j\neq k,\\
\Sigma_{\bar{X}^j}, \quad j=k.
\end{array}\right.
\end{split}
\end{equation}
i.e. $\langle \langle \Sigma_ \XB \vec{e}^{(n)}_j ,\vec{e}^{(n)}_k \rangle \rangle = \Sigma_{\bar{X} ^j} \delta _{jk}$. For $\Sigma_ \YB$ and $\Sigma_ {\XB \YB}$ similar proofs are hold. 

\medskip\noindent Now we will prove equation (\ref{sigmaXj}).  To this purpose, we will prove first that 
$$ \Sigma_\XB = \sum_{j=1} ^\infty { \Sigma _ {\bar{X}^j} \vec{e}_j}. $$
Since  $\Sigma_\XB : H^{\times n} \rightarrow H^{\times n} $, it has the following $n \times n$-matrix
\[
\Sigma_\XB =
 \begin{pmatrix}
  \Sigma _{X_1 X_1} & \Sigma _{X_1 X_2} & ... &\Sigma _{X_1 X_n}  \\
  \Sigma _{X_2 X_1} & \Sigma _{X_2 X_2} & ... &\Sigma _{X_2 X_n}  \\
  ... & ... & ... &...  \\
  \Sigma _{X_n X_1} & \Sigma _{X_n X_2} & ... &\Sigma _{X_n X_n}  
  \end{pmatrix},
\]
where $\Sigma _{X_l X_m} = \Sigma _{X_m X_l}: H \rightarrow H$, for all $ l,m =1,2,...,n$.  Let the operator $\Sigma _{X_l X_m}$ admit the following representation 
\begin{equation}
\label{Sigmaxlxm}
\Sigma _{X_l X_m} = \sum_{j=1} ^\infty {\sigma _{X_l X_m} ^j \langle h, e_j\rangle e_j,}
\end{equation}
where $\sigma _{X_l X_m} ^j$ are real numbers. 

 \medskip\noindent In view of (\ref{Sigmaxlxm}), we have 
\begin{equation}
\label{1}
 \langle \Sigma _{X_l X_m} e_j, e_k \rangle  =  \langle \sigma _{X_l X_m}^j e_j, e_k \rangle 
 = \left\{\begin{array}{lll}
0, \quad j\neq k,\\
\sigma _{X_l X_m}^j, \quad j=k.
\end{array}\right.
\end{equation}

 \medskip\noindent On the other hand, the operator $\Sigma_{\bar{X}^j}: \mathbb{R}^n \rightarrow \mathbb{R} ^n $ has the following $n \times n$-matrix
\[
\Sigma_{\bar{X}^j} =
 \begin{pmatrix}
  \mbox{cov } (x_1^j,x_1 ^j) & \mbox{cov } (x_1^j,x_2 ^j) & ... &\mbox{cov } (x_1^j,x_n ^j)  \\
    \mbox{cov } (x_2^j,x_1 ^j) & \mbox{cov } (x_2^j,x_2 ^j) & ... &\mbox{cov } (x_2^j,x_n ^j)  \\
  ... & ... & ... &...  \\
    \mbox{cov } (x_n^j,x_1 ^j) & \mbox{cov } (x_n^j,x_2 ^j) & ... &\mbox{cov } (x_n^j,x_n ^j)  
  \end{pmatrix},
\]
where $\mbox{cov } (x_l^j,x_m ^j)$ are real numbers, for all $l,m= 1,2,...,n$.

 \medskip\noindent Moreover,
\begin{equation}
\label{sigmaxlxm}
\begin{split}
\langle \Sigma _{X_l X_m} e_j, e_k \rangle &=  E ( \langle X_l , e_j \rangle \langle X_m , e_k \rangle) = E (x_1^j x_m^k) \\
& = \left\{\begin{array}{lll}
0, \quad j\neq k,\\
\mbox{cov} (x_l^j,x_m^j), \quad j=k,
\end{array}\right.
\end{split}
\end{equation}
where, for $j\neq k,$, $E (x_1^j x_m^k) =0$ follows directly from the fact that $\mbox{cov } (\bar{X}^j, \bar{X}^k) = 0$.

 \medskip\noindent Comparing (\ref{Sigmaxlxm}) with (\ref{sigmaxlxm}) gives that
$\sigma _{X_l X_m}^j= \mbox{cov} (x_l^j,x_m^j) $ for all $ l,m =1,2,...,n$.

\medskip\noindent Let $\vec{h} = (h,h,...,h) \in H^{\times n}$ with $\vec{h}^j = (\langle h,e_j \rangle,\langle h,e_j \rangle,...,\langle h,e_j \rangle) \in \mathbb{R}^{n}$. Then 
\[
\Sigma_{\bar{X}^j} \vec{h}^j =
 \begin{pmatrix}
  \sigma _{X_1 X_1} ^j \langle h, e_j\rangle & \sigma _{X_1 X_2} ^j \langle h, e_j\rangle & ... & \sigma _{X_1 X_n} ^j \langle h, e_j\rangle  \\
   \sigma _{X_2 X_1} ^j \langle h, e_j\rangle) & \sigma _{X_2 X_2} ^j \langle h, e_j\rangle& ... &\sigma _{X_2 X_n} ^j \langle h, e_j\rangle \\
  ... & ... & ... &...  \\
   \sigma _{X_n X_1} ^j \langle h, e_j\rangle & \sigma _{X_n X_2} ^j \langle h, e_j\rangle & ... &\sigma _{X_n X_n} ^j \langle h, e_j\rangle  
  \end{pmatrix}.
\]
Multiplying with $\vec{e} _j$ and summing up over $j$, it gives
\[
\begin{split}
\sum_{j=1} ^\infty {\Sigma_{\bar{X}^j} \vec{h}^j \vec{e}_j} & =
 \begin{pmatrix}
  \sum_{j=1} ^\infty {\sigma _{X_1 X_1} ^j \langle h, e_j\rangle e_j} &  \sum_{j=1} ^\infty {\sigma _{X_1 X_2} ^j \langle h, e_j\rangle e_j} & ... &  \sum_{j=1} ^\infty {\sigma _{X_1 X_n} ^j \langle h, e_j\rangle e_j}  \\
   \sum_{j=1} ^\infty {\sigma _{X_2 X_1} ^j \langle h, e_j\rangle e_j} &  \sum_{j=1} ^\infty {\sigma _{X_2 X_2} ^j \langle h, e_j\rangle e_j} & ... &  \sum_{j=1} ^\infty {\sigma _{X_2 X_n} ^j \langle h, e_j\rangle e_j}  \\
  ... & ... & ... &...  \\
   \sum_{j=1} ^\infty {\sigma _{X_n X_1} ^j \langle h, e_j\rangle e_j} &  \sum_{j=1} ^\infty {\sigma _{X_n X_2} ^j \langle h, e_j\rangle e_j} & ... &  \sum_{j=1} ^\infty {\sigma _{X_n X_n} ^j \langle h, e_j\rangle e_j}  
  \end{pmatrix}\\
  & = \begin{pmatrix}
  \Sigma _{X_1 X_1}  h & \Sigma _{X_1 X_2} h & ... &\Sigma _{X_1 X_n} h \\
  \Sigma _{X_2 X_1} h & \Sigma _{X_2 X_2} h & ... &\Sigma _{X_2 X_n} h   \\
  ... & ... & ... &...  \\
  \Sigma _{X_n X_1} h & \Sigma _{X_n X_2} h & ... &\Sigma _{X_n X_n}h   
  \end{pmatrix}\\
  &= \Sigma _\XB \vec{h}.
  \end{split}
\]

 \medskip\noindent Now, using (\ref{barxj}), we get 
 \[
\begin{split}
 \Sigma _\XB \vec{h} &= \sum_{j=1} ^\infty {\Sigma_{\bar{X}^j} \vec{h}^j \vec{e}_j}\\
 & = \sum_{j=1} ^\infty { \Sigma_{\bar{V}^j} P ^{'} (P  P ^{'}) ^{-1} (P  P ^{'}) ^{-1}  P ^{'} \vec{h}^j \vec{e}_j +\Sigma_{\bar{U}^j}\vec{h}^j \vec{e}_j }\\
 & = \Sigma _{\UB}\vec{h} + \Sigma_{\VB} P^{'}( P P^{'})^{-1} (P P^{'})^{-1} P \vec{h}.
 \end{split}
 \]
 i.e. equation (\ref{sigmaXj}) is true.  For (\ref{sigmaYj}) and (\ref{sigmaXYj}), similar proofs are hold. Since the operators $\Sigma_{\UB}$ and $ \Sigma_{\VB}$ are trace class, it follows that  $\Sigma_{\XB}, \Sigma_{\YB}$ and $\Sigma_{\XB \YB}$ are also trace class operators.  This finishes the proof.
\end{proof}

\medskip\noindent We will next compute the conditional expectation for jointly Gaussian random variables $\XB$, $\YB$.

\begin{proposition}
The conditional expectation of $\YB$ given $\XB$ is 
\begin{equation}
\label{conexp}
\begin{split}
E\left[ \YB | \XB \right] & = L\XB \\
& =  E[ \YB] +  \Sigma_ {\XB \YB} \Sigma _{\XB} ^{-1} \left( \XB - E [\XB]\right),
\end{split}
\end{equation}
provided that the opertor 
\begin{equation}
\label{T}
T= \Sigma_ {\XB \YB} \Sigma _{\XB} ^{-\frac{1}{2}}
\end{equation}
is Hilbert-Schmidt.
\end{proposition}

\begin{proof}
Following Mandelbaum, we evaluate $ E\left[ \YB | \langle \XB, \vec{e}^{(n)}_1\rangle,..., \langle \XB, \vec{e}^{(n)}_m\rangle  \right]$. We have
\[ 
\begin{split}
 E\left[ \YB | \langle \XB, \vec{e}^{(n)}_1\rangle,..., \langle \XB, \vec{e}^{(n)}_m\rangle  \right] & = E\left[ \YB | \bar{X} ^1,..., \bar{X} ^m \right]\\
 &= \sum_{j=1} ^\infty {E\left[ \bar{Y}^j | \bar{X} ^1,..., \bar{X} ^m   \right] \vec{e}^{(n)}_j}\\
 &= \sum_{j=1} ^m {E\left[ \bar{Y}^j | \bar{X} ^j \right] \vec{e}^{(n)}_j} + \sum_{j=m+1} ^\infty {E\left[ \bar{Y}^j \right] \vec{e}^{(n)}_j}.
\end{split} 
\]
The last equality is due to the independency. But relaying on the explicit form of the conditional expectation for jointly real Gaussian random variables $\bar{Y}^j, \bar{X} ^j $, we have for every $j$
\[
E\left[ \bar{Y}^j | \bar{X} ^j \right] = E[\bar{Y}^j] + \Sigma_ {\bar{X}^j \bar{Y}^j} \Sigma _{\bar{X}^j} ^{-1} \left( \bar{X}^j - E [\bar{X}^j] \right). 
\]
This yields
\[
E\left[ \YB | \langle \XB, \vec{e}_1\rangle,..., \langle \XB, \vec{e}_m\rangle  \right] = L_m \XB,
\]
where 
\[
L_m \XB= \sum_{j=1} ^\infty {E\left[ \bar{Y}^j \right] \vec{e}^{(n)}_j} + \sum_{j=1} ^m  {\Sigma_ {\bar{X}^j \bar{Y}^j} \Sigma _{\bar{X}^j} ^{-1} \left( \bar{X}^j - E [\bar{X}^j]\right) \vec{e}^{(n)}_j}.
\]
By taking the limit when $m$ goes to infinity, where by Proposition V-2-6 in \cite{Neveu} the convergence takes place in $H^{\times n}$ a.s., we obtain the formula (\ref{conexp}) of the conditional expectation. 

\medskip\noindent It remains to prove that $T$ is a Hilbert-Schmidt operator. Since $\Sigma _{\XB}$ is injective and  trace class,  the operator $\Sigma _{\XB} ^{- \frac{1}{2}}$ is Hilbert-Schmidt. Hence, $T$ is a Hilbert-Schmidt operator, since it is a product of a trace class operator with a Hilbert-Schmidt operator.
\end{proof}

The following theorem characterizes the optimal smoothing operator as the noise-to-signal.
\begin{theorem}
Under Assumptions (1) to (4) hold, the smoothing operator 
\begin{equation}
\hat{B}h= \sum_{j=1} ^\infty  {\frac{\mu_j}{\tau_j} \langle h,e_j\rangle e_j} = \Sigma_u \Sigma_v ^{-1} h, \qquad h\in L^2([0,1]),
\end{equation}
is the minimizer of the functional 
 $$
 \mathbb {J}(B) = \left\| E[\YB|\XB] - \YB(B ,\XB) \right\|^2,
 $$
 where the minimum is taken with respect to all linear, bounded  operators which satisfy the positivity condition
 $$
\langle h , B h  \rangle_{L^2([0,1])} \ge 0, \qquad h\in L^2([0,1]).
$$ 
\end{theorem}

\begin {proof}
To show that $\hat{B}$ minimizes the functional $\mathbb {J}$, let $B$  an arbitrary linear, bounded  operators which satisfy the positivity condition and prove that $\mathbb {J}(B) \geq \mathbb {J}(\hat{B})$.
\[
\begin{split}
\mathbb {J}(B) & = \left\| E[\YB|\XB] - \YB(B ,\XB) \right\|^2 \\
&= \left\| \sum_{j=1} ^\infty {E\left[ \bar{Y}^j \right] \vec{e}^{(n)}_j + \Sigma_ {\bar{X}^j \bar{Y}^j} \Sigma _{\bar{X}^j} ^{-1} \left( \bar{X}^j - E [\bar{X}^j]\right) \vec{e}^{(n)}_j} - \sum_{j=1}^\infty {\bar{Y}^j( \alpha_j, \bar{X}^j) \vec{e}^{(n)}_j} \right\|^2 \\
&= \sum_{j=1} ^\infty { \left| E\left[ \bar{Y}^j \right] + \Sigma_ {\bar{X}^j \bar{Y}^j} \Sigma _{\bar{X}^j} ^{-1} \left( \bar{X}^j - E [\bar{X}^j]\right)  - \bar{Y}^j( \alpha_j, \bar{X}^j) \right|^2 }.
\end{split}
\]
But, as a direct result of Dermoune {\it et al.} (2009), the functional
$$
\left| E\left[ \bar{Y}^j \right] + \Sigma_ {\bar{X}^j \bar{Y}^j} \Sigma _{\bar{X}^j} ^{-1} \left( \bar{X}^j - E [\bar{X}^j]\right)  - \bar{Y}^j( \alpha_j, \bar{X}^j) \right|^2
$$
has a unique minimizer $\alpha _j = \frac{\mu_j}{\tau_j}$ (which is the optimal smoothing parameter associated with the system (\ref{P-HP})).

\medskip\noindent  Hence
$$
\mathbb {J}(B) \geq \mathbb {J}(\hat{B}).
$$

\medskip\noindent it remains to show that $ \hat{B}$ is linear and bounded. The covariance operator $\Sigma _v$ is linear and bounded, since it is trace class, with $\Sigma_v > 0$. By the bounded inverse theorem \cite{Kreyszig}, $\Sigma_v ^{-1}$ exists and is bounded.  Furthermore, $\hat{B}$ is bounded, since it is a product of a trace class operator with a bounded operator.
\end{proof}

\medskip\noindent  Next we will present our main result is this paper.

\section{A consistent estimator of the noise-to-signal operator}
In this section, we propose a consistent estimator of the optimal smoothing operator $\hat {B}$, i.e. an estimator, $\mbox{est-}\hat {B} $, that converges in probability to $\hat {B}$ as the sample size tends to infinity. Due to the fact that $\hat {B}$ is the noise-to-signal ratio, it suffices to find consistent estimators of the covariance operators $\Sigma_u$ and $\Sigma _v$.

\medskip\noindent From Dermoune {\it et al.} (2008), we have, for every $j$, consistent estimators for $\mu_j$ and $\tau_j$: 
$$
\hat{\mu} _j(n) = \frac{1}{4(n-3)} \sum_{i=1}^{n-3}  { PX^j(i) PX^j(i+1)},
$$
and 
$$
\hat{\tau} _j(n) = \frac{1}{n-2}  \sum_{i=1}^{n-2} { PX^j(i) ^2 } + \frac{3}{2(n-3)}  \sum_{i=1}^{n-3} { PX^j(i) PX^j(i+1) }.
$$  

\begin{lemma}
The following statistics 
\begin{equation}
\label{sigmau-est}
\begin{split}
\hat{\Sigma}_u (n)	h &= \sum_{j=1}^\infty \hat{\mu}_j (n) \langle h,e_j \rangle e_j \\
& = \sum  _{j=1} ^\infty  \left( \frac{1}{4(n-3)} \sum_{i=1}^{n-3}  { PX^j(i) PX^j(i+1)} \right) \langle h,e_j \rangle e_j \\
& = \frac{1}{4(n-3)}  \sum_{i=1}^{n-3} { P \XB (i) P \XB (i+1) h } 
\end{split}
\end{equation}
and 
\begin{equation}
\label{sigmav-est}
\begin{split}
\hat{\Sigma}_v (n) h  &= \sum_{j=1}^\infty \hat{\tau}_j (n) \langle h,e_j \rangle e_j \\
& = \sum  _{j=1} ^\infty  \left( \frac{1}{n-2}  \sum_{i=1}^{n-2} { PX^j(i) ^2 } + \frac{3}{2(n-3)}  \sum_{i=1}^{n-3} { PX^j(i) PX^j(i+1) } \right) \langle h,e_j \rangle e_j \\
& = \frac{1}{n-2}  \sum_{i=1}^{n-2} { P \XB (i) ^2 h} + \frac{3}{2(n-3)}  \sum_{i=1}^{n-3} { P \XB (i) P \XB (i+1) h } 
\end{split}
\end{equation}
based on the time series of observation $P \XB$, are consistent estimators of the covariance operator $\Sigma_u$ and $\Sigma_v$, respectively.
\end{lemma}


\begin{proof}

We only carry out the proof of consistency for the estimator $\hat{\Sigma}_u$. A similar proof for $\hat{\Sigma}_v$ holds. 
by Proposition 2 in Dermoune {\it et al.}, we have after simple calculations 
\[
\begin{split}
\mbox{var} (\hat{\mu}_j(n))  = \frac{1}{16(n-3)^2} & \left[ (n-3) \left( \tau_j^2 +12 \tau_j \mu_j +52 \mu_j ^2 \right) \right. \\
& \qquad  \left. + 2(n-4)  \left( \tau_j \mu_j +22 \mu_j ^2 \right) + 2 (n-5)  \mu_j ^2. \right]
\end{split}
\]
Then it holds for arbitrary $h \in L^2( [0,1]) $ that
\[
\begin{split}
E \left\| \hat{\Sigma} _u (n) h - \Sigma _u h \right\| ^2 & = E \left\| \sum_{j=1} ^\infty {( \hat{\mu} _j (n) - \mu_ j)} \langle h,e_j \rangle e_j \right\| ^2 \\
& \leq \sum_{j=1} ^\infty {E \left| \hat{\mu} _j (n) - \mu_ j \right| ^2 \left\| h\right\|^2 }\\
& =  \left\| h\right\|^2 \sum_{j=1} ^\infty {\mbox{var} (\hat{\mu}_j(n))}  \rightarrow 0, \quad \mbox{ as } n \rightarrow \infty
\end{split}
\]
with 
\[ 
\begin{split} 
\sum_{j=1} ^\infty \mbox{var} (\hat{\mu}_j(n)) = \frac{1}{16(n-3)^2}\sum_{j=1} ^\infty & \left[ (n-3) \left( \tau_j^2 +12 \tau_j \mu_j +52 \mu_j ^2 \right) \right. \\
& \qquad  \left. + 2(n-4)  \left( \tau_j \mu_j +22 \mu_j ^2 \right) + 2 (n-5)  \mu_j ^2  \right]
\end{split}
\] is finite since $\Sigma_u$ and $\Sigma_v$ are trace class operators.
\end{proof}
The following theorem gives a consistent estimators of  the optimal smoothing operator $\hat{B}$. It is the functional time series extension of Theorem 1 in Dermoune {\it et al.} (2009).
\begin{theorem}
The optimal smoothing operator $\hat{B}$ admits the following consistent estimator 

\begin{equation}
\mbox{est-} \hat{B}(n)  h = - \frac{1}{4} \sum_{j=1} ^\infty { \left( \frac{3}{2} + \frac{(n-3) \sum_{i=1} ^{n-2} { PX^j(i) ^2}}{ (n-2) {\sum_{i=1}^{n-3} { PX^j(i) PX^j(i+1)}}} \right)^{-1} \langle h,e_j \rangle e_j }
\end{equation}
based on the time series of observation $P \XB$.
\end{theorem}


\section*{Acknowledgment}
I would like to thank Professor Boualem Djehiche and Associate Professor Astrid Hilbert for the fruitful discussions, the guidance and the advice they have provided me.


\end{document}